\documentclass[3p,12pt]{elsarticle}

\usepackage[T1]{fontenc}
\usepackage{lmodern}
\usepackage{amsmath,amssymb,amsthm,mathtools}
\usepackage{enumitem}
\usepackage{booktabs}
\usepackage{microtype}
\usepackage{needspace}
\usepackage{tikz}
\usetikzlibrary{positioning}
\usepackage[hidelinks]{hyperref}
\journal{Topology and its Applications}
\biboptions{sort&compress}

\newtheorem{theorem}{Theorem}[section]
\newtheorem{proposition}[theorem]{Proposition}
\newtheorem{lemma}[theorem]{Lemma}
\newtheorem{corollary}[theorem]{Corollary}
\theoremstyle{definition}

\theoremstyle{remark}

\newtheorem{question}[theorem]{Question}

\newcommand{\Is}{\operatorname{Is}}
\newcommand{\SI}{\operatorname{SI}}
\newcommand{\SIcpt}{\operatorname{SI}_{\mathrm{cpt}}}
\newcommand{\SIrw}{\operatorname{SI}_{\mathrm{RW}}}
\newcommand{\SIcoh}{\operatorname{SI}_{\mathrm{coh}}}
\newcommand{\OO}{\mathcal O}
\newcommand{\id}{\operatorname{id}}

\newcommand{\tprod}{\tau_{\mathrm{prod}}}
\newcommand{\tpt}{\tau_{\mathrm{pt}}}
\newcommand{\up}{\mathord{\uparrow}}
\newcommand{\da}{\mathord{\downarrow}}

\begin{document}

\begin{frontmatter}

\title{Scott--Isbell Coincidence for Continuous Dcpos beyond Bicompleteness}

\author[scu]{Yuxu Chen}
\ead{chenyuxu@scu.edu.cn}
\address[scu]{School of Mathemtaics, Sichuan University, Chengdu 610015, China}

\author[scu]{Yunjie Wen}

\author[yc]{Xiaoyong Xi}

\address[yc]{School of Mathematics and Statistics, Yancheng Teachers University, Yancheng 224002,  China}

\begin{abstract}
Lawson and Mislove posed the following problem in 1990 as Problem~535 in \emph{Open Problems in Topology}: for a core-compact space \(X\) and a dcpo \(P\) equipped with its Scott topology, under what conditions on \(P\) do the Isbell and Scott topologies on \(C(X,P)\) agree?
It was proved that, for a nonempty bicomplete continuous dcpo \(P\), the Isbell and Scott topologies on \(C(X,P)\) coincide for every core-compact space \(X\) if and only if \(P\) is bounded complete; for every compact core-compact space \(X\) if and only if \(P\) is conditionally bounded complete; and for every RW-space \(X\) if and only if \(P\) is a pointed continuous \(L\)-domain.
We remove the bicompleteness assumption from all three classifications by combining the forbidden-retract theorem of Jia, Jung and Li with a separation theorem for powers of downward well-ordered chains and suitable Alexandrov test spaces. 

\end{abstract}


\end{frontmatter}

\section{Introduction}
\label{sec:introduction}

Let $X$ be a topological space and let $L$ be a dcpo equipped with its Scott
topology. We write $C(X,L)$ for the set of continuous maps from $X$ to the
Scott space $\Sigma L$, ordered pointwise. This function space carries two
canonical topologies. Its order gives the Scott topology
$\sigma(C(X,L))$, while the topologies of $X$ and $L$ give the Isbell
topology $\Is(C(X,L))$ \cite{Isbell1975}. One always has
\[
  \Is(C(X,L))\subseteq\sigma(C(X,L)),
\]
and the problem is to determine when equality holds.

Lawson and Mislove considered two related versions of this comparison problem \cite{LawsonMislove1990}. The unrestricted version asks, for a topological space $X$ and a dcpo $L$, when the two topologies on $C(X,L)$ agree. In Problem~535 they singled out the case in which $X$ is core compact and asked for conditions on the dcpo $L$. We study three formulations of this problem, characterizing the continuous dcpos $L$ for which coincidence holds for every space $X$ in a prescribed class. This formulation is in line with the maximal-class approach of Lawson and Xu \cite{LawsonXu2003} and its Scott--Isbell development by Xi, Liu, Ho and Zhao \cite{XiLiuHoZhao2017}.

Liu and Liang obtained early results for continuous $L$-domains
\cite{LiuLiang1996}. Xi and Yang later proved that, among bicomplete
continuous dcpos, coincidence for all core-compact sources is equivalent to
bounded completeness, while coincidence for all compact core-compact sources
is equivalent to conditional bounded completeness \cite{XiYang2014}. Xi
proved the corresponding result for RW-spaces: in the same bicomplete ambient
class, coincidence for all RW-spaces characterizes pointed continuous
$L$-domains \cite{Xi2005}. Xi, Liu, Ho and Zhao subsequently studied
maximal source--target pairings for Scott--Isbell coincidence, including the
compact core-compact and RW cases, again with targets restricted to
bicomplete domains \cite{XiLiuHoZhao2017}. We remove the bicompleteness
assumption from all three target-side classifications.

For a continuous dcpo $L$, let $\SI(L)$ mean that
\[
  \Is(C(X,L))=\sigma(C(X,L))
  \quad\text{for every core-compact space }X,
\]
let $\SIcpt(L)$ denote the analogous condition with $X$ restricted to
compact core-compact spaces, and let $\SIrw(L)$ denote coincidence for every
RW-space. Our main results are the following.

\begin{theorem}\label{thm:main-corecompact}
Let $L$ be a nonempty continuous dcpo. Then $\SI(L)$ holds if and only if $L$ is
bounded complete.
\end{theorem}

\begin{theorem}\label{thm:main-compact}
Let $L$ be a nonempty continuous dcpo. Then $\SIcpt(L)$ holds if and only if $L$ is
conditionally bounded complete.
\end{theorem}

\begin{theorem}\label{thm:main-rw}
Let $L$ be a nonempty continuous dcpo. Then $\SIrw(L)$ holds if and only if $L$ is a
pointed continuous $L$-domain.
\end{theorem}

The three proofs are based on a common mechanism. Jia, Jung and Li showed
that non-bicompleteness of a meet-continuous sober dcpo is witnessed by one
of three standard Scott-continuous retracts built from a downward
well-ordered chain without a bottom element \cite{JiaJungLi2019}. We first
prove that a suitable power of every such chain has Scott topology strictly
finer than the product of the coordinate Scott topologies. Discrete and
Alexandrov source spaces then transfer this separation to function spaces.
For compact sources the argument is applied after adjoining a new bottom
element, and two finite retracts eliminate the remaining failures of bounded
completeness. For RW-spaces we use capped Alexandrov detectors and restriction--extension
retractions.

The paper is organized as follows. Section~\ref{sec:preliminaries} collects
the definitions, standard tools, and published results used later.
Section~\ref{sec:separation-detectors} establishes the power separation and
its transfer to Alexandrov function spaces. Section~\ref{sec:classifications}
proves the three main classifications. 

\section{Preliminaries}
\label{sec:preliminaries}

\subsection{Ordered and topological structures}

A subset $D$ of a poset is \emph{directed} if every two elements of $D$ have an upper bound
in $D$, and it is \emph{filtered} if every two elements have a lower bound in $D$. All directed and filtered subsets in this paper are nonempty. 
A \emph{dcpo} is a poset in which every directed subset has a supremum.
For a subset $A$ of a poset $P$, write
\[
  \up A=\{y:\exists x\in A,\ x\leq y\},
  \qquad
  \da A=\{y:\exists x\in A,\ y\leq x\},
\]
and use $\up x$ and $\da x$ for singletons. The order dual of
$P$ is denoted by $P^{\mathrm{op}}$, and a chain is a set of pairwise
comparable elements.

For $u,x$ in a dcpo $P$, one writes $u\ll x$ if, whenever $D\subseteq P$ is
directed and $x\leq\sup D$, some $d\in D$ satisfies $u\leq d$. An element
$k$ is \emph{compact} if $k\ll k$. The dcpo $P$ is \emph{continuous} if
$\{u:u\ll x\}$ is directed with supremum $x$ for every $x\in P$, and it is
\emph{algebraic} if the compact elements below each $x$ form such a directed
set.

A subset $U\subseteq P$ is \emph{Scott open} if it is upward closed and
inaccessible by directed suprema. The Scott-open sets form the Scott topology
$\sigma(P)$, and $\Sigma P=(P,\sigma(P))$ denotes the Scott space. A set is
Scott closed precisely when it is a lower set closed under directed suprema. A map
between dcpos is \emph{Scott continuous} if it is monotone and preserves
directed suprema. A dcpo is \emph{pointed} if it has a least element $\bot$.
The lifting \(P_\bot\) of \(P\) is obtained by adjoining a new least element to \(P\).

A dcpo is \emph{bicomplete} if every filtered subset has an infimum. It is
\emph{bounded complete} if every subset having an upper bound, including the empty subset, has a supremum, and it is \emph{conditionally bounded
complete} if this requirement is imposed only on nonempty subsets. A dcpo $P$ is an \emph{$L$-dcpo} if every principal ideal $\da x$ is a complete lattice, and a continuous $L$-dcpo is called an \emph{$L$-domain}.

A dcpo $R$ is a \emph{Scott-continuous retract} of a dcpo $P$ if there are
Scott-continuous maps $s:R\to P$ and $r:P\to R$ such that
$r\circ s=\id_R$. If $U$ is a Scott-open subset of a dcpo $P$, we regard $U$ as an ordered
subspace of $P$ with the inherited order. It is itself a dcpo: if
$D\subseteq U$ is directed, then its supremum computed in $P$ lies in $U$ and is also its supremum there.
Its intrinsic Scott topology is the subspace topology inherited from
$\sigma(P)$.

For a topological space $X$, write $\OO(X)$ for its open-set lattice ordered
by inclusion. The space is \emph{core compact} if $\OO(X)$ is a continuous
dcpo. A subset is \emph{saturated} if it is an intersection of open sets,
and $X$ is \emph{coherent} if the intersection of any two compact saturated
subsets is compact. The specialization preorder is denoted by $\sqsubseteq_X$ and is defined by
\[
  x\sqsubseteq_X y
  \quad\Longleftrightarrow\quad
  \text{every open neighbourhood of }x\text{ contains }y.
\]

For a dcpo $P$, let $C(X,P)$ be the set of continuous maps
$X\to\Sigma P$, ordered pointwise. Directed suprema in $C(X,P)$ are computed pointwise \cite[Proposition~II-3.15(i)]{GierzEtAl2003}. The same argument applies to arbitrary $X$: for a directed $\mathcal D\subseteq C(X,P)$, its pointwise supremum satisfies $(\sup\mathcal D)^{-1}(V)=\bigcup_{f\in\mathcal D}f^{-1}(V)$ for every $V\in\sigma(P)$, and is therefore continuous. The \emph{pointwise topology} $\tpt(C(X,P))$ is generated by the sets
\[
  \bigcap_{i=1}^{n}
  \{f\in C(X,P):f(x_i)\in V_i\},
\]
where $x_1,\ldots,x_n\in X$ and $V_1,\ldots,V_n\in\sigma(P)$. 
For a Scott-open family $\mathcal H\subseteq\OO(X)$ and
$V\in\sigma(P)$, put
\begin{equation}\label{eq:isbell-subbasic}
  [\mathcal H,V]
  =\{f\in C(X,P):f^{-1}(V)\in\mathcal H\}.
\end{equation}
These sets generate the \emph{Isbell topology}
$\Is(C(X,P))$. For a set $I$, the power $P^I$ is ordered pointwise and
\[
  \tprod(P^I)=\prod_{i\in I}\sigma(P)
\]
denotes the product of the coordinate Scott topologies.

\subsection{Basic function-space properties}

Postcomposition is continuous for both function-space topologies. A map between function spaces is called \emph{Isbell continuous} if it is continuous when both spaces carry their Isbell topologies.

\begin{lemma}\label{lem:basic-target-tools}
Let $h:P\to Q$ be Scott continuous. For every topological space $X$, the map
\[
  h_*:C(X,P)\to C(X,Q),\qquad h_*(f)=h\circ f,
\]
is both Scott continuous and Isbell continuous.
\end{lemma}

\begin{proof}
Since $h$ is Scott continuous, $h\circ f$ is continuous for every
$f\in C(X,P)$, so $h_*$ is well defined.

Let $\mathcal D\subseteq C(X,P)$ be directed. Directed suprema in
$C(X,P)$ are computed pointwise, and therefore, for every $x\in X$,
\[
 h_*\Bigl(\sup \mathcal D\Bigr)(x)
 =h\Bigl(\bigl(\sup\mathcal D\bigr)(x)\Bigr)
 =h\Bigl(\sup_{f\in\mathcal D}f(x)\Bigr)
 =\sup_{f\in\mathcal D}h(f(x))
 =\Bigl(\sup_{f\in\mathcal D}h_*(f)\Bigr)(x),
\]
Hence $h_*$ preserves directed suprema and is Scott continuous.

For the Isbell topology, let $[\mathcal H,V]$ be a subbasic open set
of $C(X,Q)$, where $\mathcal H$ is a Scott-open family of open
subsets of $X$ and $V$ is Scott open in $Q$. Then
\[
 h_*^{-1}[\mathcal H,V]
 =[\mathcal H,h^{-1}(V)].
\]
Since $h$ is Scott continuous, $h^{-1}(V)$ is Scott open in $P$.
Thus $[\mathcal H,h^{-1}(V)]$ is Isbell open, and $h_*$ is Isbell
continuous.
\end{proof}

The following classical inclusion follows from \cite[Definition~II-3.12 and Lemmas~II-4.2(ii), II-4.3(i)]
{GierzEtAl2003}, and it is also used in \cite{XiLiuHoZhao2017}.

\begin{proposition}
\label{prop:isbell-below-scott}
For every topological space $X$ and every dcpo $L$,
\(
  \Is(C(X,L)) \) is coarser than \(\sigma(C(X,L)).
\)
\end{proposition}


The Isbell--Scott coincidence property on the function space is stable under two natural operations on the target: passing to Scott-continuous retracts and, for compact sources, passing to Scott-open sub-dcpos.

\Needspace{10\baselineskip}
\begin{proposition}\label{prop:target-inheritance}
  Let $L$ be a dcpo.
\begin{enumerate}[label=(\roman*)]
\item Let $R$ be a Scott-continuous retract of $L$. If
\(
  \operatorname{Is}(C(X,L))=\sigma(C(X,L)),
\)
then the same equality holds with $R$ in place of $L$.

\item Let $U$ be a Scott-open sub-dcpo of $L$. If the Isbell and Scott
topologies on $C(X,L)$ coincide for a compact space $X$, then they also
coincide on $C(X,U)$.

\end{enumerate}
\end{proposition}

\begin{proof}
For (i), choose Scott-continuous maps $s:R\to L$ and $r:L\to R$ with
$r\circ s=\id_R$. If $W\subseteq C(X,R)$ is Scott open, then
$r_*^{-1}(W)$ is Scott open in $C(X,L)$ and hence Isbell open under the
hypothesis. Lemma~\ref{lem:basic-target-tools} gives
\[
  W=s_*^{-1}\bigl(r_*^{-1}(W)\bigr),
\]
so $W$ is Isbell open. Proposition~\ref{prop:isbell-below-scott} gives the
reverse inclusion.

For (ii), compactness makes $X$ a compact element of $\OO(X)$, so
$\{X\}$ is Scott open and
\[
  C(X,U)=[\{X\},U]
\]
is both Isbell open and Scott open in $C(X,L)$. Its intrinsic Scott topology is therefore the relative Scott topology. If $V$ is intrinsically Scott open in $U$, then
$V=U\cap V'$ for some $V'\in\sigma(L)$, and
\[
  [\mathcal H,V]_{C(X,U)}
  =C(X,U)\cap[\mathcal H,V']_{C(X,L)}.
\]
Conversely, the restriction of $[\mathcal H,V']$ to $C(X,U)$ is
$[\mathcal H,U\cap V']_{C(X,U)}$. Thus the intrinsic Isbell topology is the relative Isbell topology, and coincidence restricts to $C(X,U)$.
\end{proof}

For a poset $P$, the \emph{Alexandrov topology} consists of all upper subsets, which is  generated by
the principal upper sets $\up x$, $x\in P$, and the corresponding space is called an Alexandrov space.  A map from an Alexandrov space to a Scott space is continuous exactly when it is monotone.

Xi, Xu and Lawson proved that the pointwise, compact-open and Isbell
topologies on $C(X,Y)$ coincide for every topological space $Y$
whenever $X$ is a quasicontinuous space \cite[Lemma~5.3]{XiXuLawson2016}.
Every Alexandrov space is quasicontinuous and core compact. Thus we have the following immediate consequence.

\begin{lemma}\label{lem:alexandrov}
If $X$ is an Alexandrov space and $L$ is a dcpo, then $X$ is core compact and
\[
  \Is(C(X,L))=\tpt(C(X,L)).
\]
\end{lemma}

A \emph{decomposition} of a nonempty open set $V$ is a family of pairwise
disjoint nonempty open sets whose union is $V$. Let $V\subseteq U$ be
nonempty open sets. A decomposition $\mathcal A$ of $V$ is
\emph{relatively maximal to $U$} if, for every decomposition $\mathcal B$
of $U$ and every $B\in\mathcal B$ meeting $V$,
\[
  B\cap V
  =\bigcup\{A\in\mathcal A:A\cap B\neq\varnothing\}.
\]
A topological space $X$ has \emph{property RW} if, whenever
$V_i\ll U_i$ in $\OO(X)$ for $1\leq i\leq k$, where $k\geq1$, and
$\bigcap_{i=1}^{k}V_i\neq\varnothing$, the open set
$\bigcap_{i=1}^{k}V_i$ has a finite decomposition relatively maximal to
$\bigcap_{i=1}^{k}U_i$. A core-compact space with property RW is called an
\emph{RW-space} \cite[Definition~2.4]{KouLuo2003}.

\subsection{Retract forbidden structures and bicomplete classifications}

A nonempty closed set is \emph{irreducible} if it is not the union of two proper
closed subsets. 
A $T_0$ space is \emph{sober} if every nonempty irreducible
closed subset is the closure of a unique point. A dcpo is called sober when its Scott
space is sober.
A dcpo $P$ is \emph{meet continuous} if, whenever $D\subseteq P$ is
directed and $x\leq\sup D$, every Scott-open neighbourhood of $x$ meets
$\da x\cap\da D$. Every continuous dcpo is sober and meet
continuous \cite[Proposition~III-3.7(i), Theorem~III-2.11]{GierzEtAl2003}.

The starting point for the first two classifications is the following theorem.

\begin{theorem}[{\cite{XiYang2014}}]\label{thm:xi-yang}
Let $L$ be a nonempty bicomplete continuous dcpo.
\begin{enumerate}[label=(\roman*)]
\item $\SI(L)$ holds if and only if $L$ is bounded complete.
\item $\SIcpt(L)$ holds if and only if $L$ is conditionally bounded complete.
\end{enumerate}
\end{theorem}

For RW-spaces we use the following bicomplete classification. 

\begin{theorem}[{\cite{Xi2005}}]\label{thm:xi-rw}
Let $L$ be a nonempty bicomplete continuous dcpo. Then $\SIrw(L)$ holds if and only
if $L$ is a pointed continuous $L$-domain.
\end{theorem}

We next introduce the standard retract obstructions. A nonempty chain $W$ is
\emph{downward well ordered} if every nonempty subset has a greatest element.
For such a chain without a bottom element, define
\[
  K(W)=W\cup\{a,b\},
\]
where $a$ and $b$ are incomparable and lie below every element of $W$; let
$K(W)_\bot$ be its lifting. The following theorem is the main result of Jia, Jung and Li. It is the key to removing the bicompleteness assumption from the three classifications.

\begin{theorem}[{\cite[Corollaries~4.3 and~4.4]{JiaJungLi2019}}]
\label{thm:JJL-bicomplete}
\begin{enumerate}[label=(\roman*)]
\item If a meet-continuous sober dcpo $L$ is not bicomplete, then one of
$W$, $K(W)$, and $K(W)_\bot$ is a Scott-continuous retract of $L$ for some
downward well-ordered chain $W$ without a bottom element.
\item If a pointed sober dcpo $P$ is not bicomplete, then $K(W)_\bot$ is a
Scott-continuous retract of $P$ for some such $W$.
\end{enumerate}
\end{theorem}

Two finite pointed obstructions will be needed for compact sources. Let
\[
  \mathsf B_5=\{\bot,u,v,a,b\},
\]
where \(\bot\) is the least element, \(u\) and \(v\) are incomparable,
\(a\) and \(b\) are incomparable, and
\(
  u,v<a,b.
\)
Let $\mathsf B_5^\top$ be obtained from $\mathsf B_5$ by adjoining a new
greatest element $\top$.
Finally, put
\[
  \mathsf B_4=\mathsf B_5\setminus\{\bot\},
  \qquad
  \mathsf B_4^\top=\mathsf B_5^\top\setminus\{\bot\}.
\]

\begin{figure}[t]
\centering
\begin{minipage}{0.47\linewidth}
\centering
\begin{tikzpicture}[x=0.95cm,y=0.75cm,
 every node/.style={circle,draw,inner sep=1.2pt,minimum size=5.5mm}]
\node (bot) at (0,0) {$\bot$};
\node (u) at (-1,1) {$u$};
\node (v) at (1,1) {$v$};
\node (a) at (-1,2) {$a$};
\node (b) at (1,2) {$b$};
\node (top) at (0,3) {$\top$};
\draw (bot)--(u) (bot)--(v);
\draw (u)--(a) (u)--(b) (v)--(a) (v)--(b);
\draw (a)--(top) (b)--(top);
\end{tikzpicture}

$\mathsf B_5^\top$
\end{minipage}
\hfill
\begin{minipage}{0.47\linewidth}
\centering
\begin{tikzpicture}[x=0.95cm,y=0.85cm,
 every node/.style={circle,draw,inner sep=1.2pt,minimum size=5.5mm}]
\node (bot) at (0,0) {$\bot$};
\node (u) at (-1,1) {$u$};
\node (v) at (1,1) {$v$};
\node (a) at (-1,2) {$a$};
\node (b) at (1,2) {$b$};
\draw (bot)--(u) (bot)--(v);
\draw (u)--(a) (u)--(b) (v)--(a) (v)--(b);
\end{tikzpicture}

$\mathsf B_5$
\end{minipage}
\caption{The finite pointed retract obstructions used for compact sources.}
\label{fig:finite-obstructions}
\end{figure}
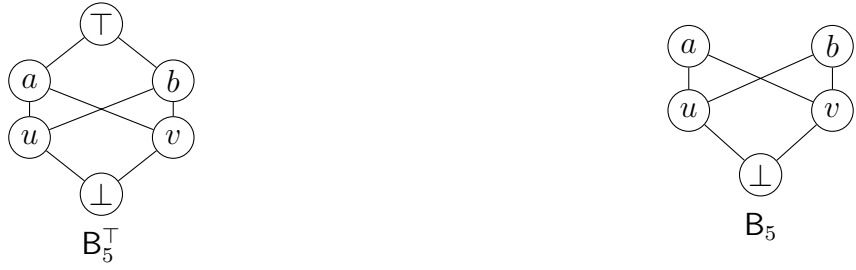

\begin{theorem}[{\cite[Theorem~5.2]{JiaJungLi2019}}]
\label{thm:JJL-nonL}
Let $P$ be a pointed sober bicomplete dcpo. If $P$ is not an $L$-dcpo, then
$\mathsf B_5^\top$ is a Scott-continuous retract of $P$.
\end{theorem}

The poset $\mathsf B_5$ is the standard continuous $L$-dcpo that is not
bounded complete~\cite[Figure 11]{AbramskyJung1994}. We shall also use the standard fact that every bounded-complete continuous
dcpo is bicomplete \cite[Proposition~4.1.2]{AbramskyJung1994}.

\begin{theorem}[{\cite[Exercise~4.3.11(3c), p.~66]{AbramskyJung1994}}]
\label{thm:AJ-L-not-BC}
Let $P$ be a pointed continuous $L$-dcpo. If $P$ is not bounded complete,
then $\mathsf B_5$ is a Scott-continuous retract of $P$.
\end{theorem}

\section{Separation constructions and Alexandrov test spaces}
\label{sec:separation-detectors}

Theorem~\ref{thm:JJL-bicomplete} reduces the failure of bicompleteness
to three retract obstructions,
\(
  W,\ K(W),\ K(W)_\bot,
\)
where \(W\) is a downward well-ordered chain without a least element.
In this section we show that these obstructions fail the Scott--Isbell
coincidence properties needed later. We first treat the basic countable
case, where the separation can be described explicitly. We then extend
the same argument to an arbitrary \(W\), which requires
a cardinal argument.

Throughout this section, elements of a power \(W^I\) are written as
functions \(f:I\to W\), with coordinate values \(f(i)\), and
\(\preceq\) denotes the pointwise order on the power. We write
\(\top\) for the greatest element of \(W\), when it exists, and
\(\top_I\) for the greatest element of \(W^I\), that is, the
constant function \(\top_I(i)=\top\).

\subsection{Countable obstructions}

The simplest instance is the descending chain
\(
  \mathbb N^{\mathrm{op}}=\{\top>1>2>\cdots\},
  \qquad \top=0.
\)
The corresponding forms of the three retract obstructions are shown in
Figure~\ref{fig:countable-obstructions}.

\begin{figure}[!htbp]
\centering
\begin{tikzpicture}[
  x=1.25cm,y=0.72cm,
  every node/.style={inner sep=1.5pt}
]

\node at (-3.8,3.8) {\(\mathbb N^{\mathrm{op}}\)};
\node (n0) at (-3.8,3.1) {\(\top\)};
\node (n1) at (-3.8,2.45) {\(1\)};
\node (n2) at (-3.8,1.80) {\(2\)};
\node at (-3.8,1.15) {\(\vdots\)};
\draw (n0)--(n1)--(n2);

\node at (0,3.8) {\(K(\mathbb N^{\mathrm{op}})\)};
\node (k0) at (0,3.1) {\(\top\)};
\node (k1) at (0,2.45) {\(1\)};
\node (k2) at (0,1.80) {\(2\)};
\node at (0,1.15) {\(\vdots\)};
\node (a) at (-0.55,0.2) {\(a\)};
\node (b) at (0.55,0.2) {\(b\)};
\draw (k0)--(k1)--(k2);
\draw[dashed] (0,0.7)--(a);
\draw[dashed] (0,0.7)--(b);

\node at (4.2,3.8) {\(K(\mathbb N^{\mathrm{op}})_\bot\)};
\node (l0) at (4.2,3.1) {\(\top\)};
\node (l1) at (4.2,2.45) {\(1\)};
\node (l2) at (4.2,1.80) {\(2\)};
\node at (4.2,1.15) {\(\vdots\)};
\node (la) at (3.65,0.2) {\(a\)};
\node (lb) at (4.75,0.2) {\(b\)};
\node (bot) at (4.2,-0.45) {\(\bot\)};
\draw (l0)--(l1)--(l2);
\draw[dashed] (4.2,0.7)--(la);
\draw[dashed] (4.2,0.7)--(lb);
\draw (la)--(bot);
\draw (lb)--(bot);

\end{tikzpicture}
\caption{The basic countable forms of the three retract obstructions,
with \(\top=0\). The dashed lines indicate that \(a\) and \(b\) lie below every element
of the descending chain.}
\label{fig:countable-obstructions}
\end{figure}
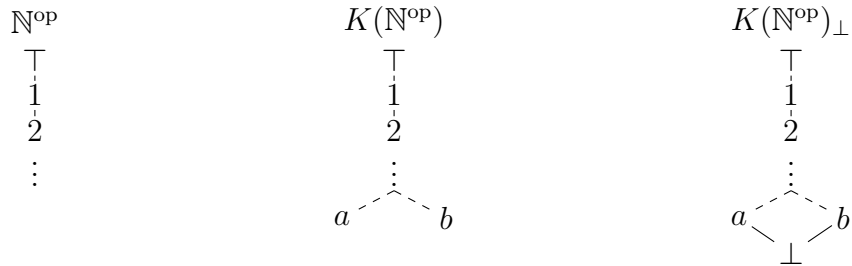

\Needspace{11\baselineskip}
\begin{proposition}\label{prop:countable-prototype}
Let \(\mathbb N^{\mathrm{op}}\) denote the natural numbers with the
reverse of their usual order. If \(I\) is countable, then
\[
  \tprod\bigl((\mathbb N^{\mathrm{op}})^I\bigr)
  =
  \sigma\bigl((\mathbb N^{\mathrm{op}})^I\bigr),
\]
whereas if \(I\) is uncountable, then
\[
  \tprod\bigl((\mathbb N^{\mathrm{op}})^I\bigr)
  \subsetneq
  \sigma\bigl((\mathbb N^{\mathrm{op}})^I\bigr).
\]
\end{proposition}

\begin{proof}
Put \(P=(\mathbb N^{\mathrm{op}})^I\). In this proof, inequalities
between natural numbers refer to their usual order. Thus the pointwise
order on \(P\) is given by
\[
  f\preceq g
  \quad\Longleftrightarrow\quad
  f(i)\geq g(i)\quad\text{for every }i\in I.
\]
Recall that \(\top=0\) in \(\mathbb N^{\mathrm{op}}\).
Since coordinate projections are Scott continuous, the product topology
is always contained in the Scott topology.

\medskip
\noindent
\emph{The countable case.}
If \(I\) is finite, every element of \(P\) is compact and its
principal upper set is product open, so the two topologies agree.
Suppose that \(I\) is countably infinite and identify it with
\(\{1,2,3,\ldots\}\). Let \(F\subseteq P\) be Scott closed and
let \(z\in P\) belong to its closure of the product topology.

For \(n\geq1\), the set
\[
  U_n=\{f\in P:f(i)\leq z(i)\text{ for }1\leq i\leq n\}
\]
is a basic product neighbourhood of \(z\), since
\(\up_{\mathbb N^{\mathrm{op}}}z(i)
=\{k\in\mathbb N:k\leq z(i)\}\) is Scott open.
Choose \(a_n\in F\cap U_n\).
For \(n\geq1\), define \(d_n\in P\) by
\[
  d_n(i)=
  \begin{cases}
    z(i),&i\leq n,\\[1mm]
    \max\bigl(\{z(i)\}\cup\{a_k(i):1\leq k<i\}\bigr),&i>n,
  \end{cases}
\]
where the maximum is taken in the usual order on \(\mathbb N\).
For \(i\leq n\), we have \(a_n(i)\leq z(i)=d_n(i)\); for
\(i>n\), the term \(a_n(i)\) occurs in the maximum defining
\(d_n(i)\). Hence \(d_n\preceq a_n\), and therefore \(d_n\in F\).

Moreover, \(d_n\preceq d_{n+1}\). These functions agree except
possibly at coordinate \(n+1\), where
\(d_n(n+1)\geq z(n+1)=d_{n+1}(n+1)\) in the usual order.
For every fixed \(i\), one has \(d_n(i)=z(i)\) once \(n\geq i\).
Hence \(\sup_n d_n=z\), so Scott closedness gives \(z\in F\).
Thus every Scott-closed set is product closed, and the two topologies
coincide.

\medskip
\noindent
\emph{The uncountable case.}
Suppose that \(I\) is uncountable. We construct a Scott-closed set
\(F\subseteq P\) which does not contain \(\top_I\) but meets every basic
product neighbourhood of \(\top_I\).
For each nonempty finite subset \(J\subseteq I\), define
\[
  g_J(i)=
  \begin{cases}
    \top,&i\in J,\\
    |J|,&i\notin J,
  \end{cases}
  \qquad
  F=\da\{g_J:J\subseteq I\text{ is nonempty and finite}\},
\]
where the lower closure is taken with respect to \(\preceq\).

To describe membership in \(F\), for \(n\geq1\) and \(f\in P\)
put
\[
  E_n(f)=\{i\in I:f(i)<n\},
  \qquad
  F_n=\{f\in P:|E_n(f)|\leq n\}.
\]
Thus \(F_n\) consists of the functions whose value is less than
\(n\) at no more than \(n\) coordinates. We have
\(
  F=\bigcup_{n\geq1}F_n.
\)
Indeed, if \(f\preceq g_J\) and \(n=|J|\), then
\(f(i)\geq n\) outside \(J\), so \(E_n(f)\subseteq J\).
Conversely, if \(f\in F_n\), enlarge \(E_n(f)\) to a set
\(J\subseteq I\) of cardinality \(n\). Then \(f(i)\geq n\)
outside \(J\), and \(f\preceq g_J\).

We now prove that \(F\) is Scott closed. 
We first observe that each \(F_n\) is Scott closed. Clearly \(F_n\) is
a lower set with respect to \(\preceq\). Let \(D\subseteq F_n\) be
nonempty and directed, and let \(s=\sup D\). We have, coordinatewise,
\(
  s(i)=\min\{d(i):d\in D\},
\)
and hence
\(
  E_n(s)
  =
  \bigcup_{d\in D}E_n(d).
\)
Moreover, the family \(\{E_n(d):d\in D\}\) is directed under inclusion.
Each \(E_n(d)\) has at most \(n\) elements. Therefore their directed
union also has at most \(n\) elements. Thus
\(
  |E_n(s)|\leq n,
\)
so \(s\in F_n\). Hence every \(F_n\) is Scott closed.

The point is now to show that, although
\(
  F=\bigcup_{n\geq1}F_n
\)
is an infinite union, a given directed subset of \(F\) is eventually
contained in a single \(F_n\).

Let \(D\subseteq F\) be nonempty and directed, and choose \(d_0\in D\).
Since every value \(d_0(i)\) is a natural number,
\(
  I=\bigcup_{m\geq1}E_m(d_0).
\)
As \(I\) is uncountable, at least one of the sets \(E_m(d_0)\) is
infinite. Fix such an \(m\), and consider the cofinal directed subset
\(
  D_0=\{d\in D:d_0\preceq d\}.
\)
If \(d\in D_0\), then
\(
  d(i)\leq d_0(i),
\) and consequently
\(
  E_m(d_0)\subseteq E_m(d).
\)
Thus \(E_m(d)\) is infinite. In particular, \(d\notin F_n\) for every
\(n\geq m\), since
\(
  E_m(d)\subseteq E_n(d)
\).  Since \(d\in F\), it follows that
\[
  D_0\subseteq F_1\cup\cdots\cup F_{m-1}.
\]

We use the elementary fact that a directed set covered by finitely many
lower sets must be contained in one of them. Then, there exists \(n_*<m\) such that
\(
  D_0\subseteq F_{n_*}.
\)
Since \(F_{n_*}\) is Scott closed and \(D_0\) is cofinal in \(D\),
\(
  \sup D=\sup D_0\in F_{n_*}\subseteq F.
\)
Hence \(F\) is Scott closed.

\emph{Finally, \(F\) is not closed with respect to the product topology.}
For every \(n\geq1\), we have \(E_n(\top_I)=I\), so
\(\top_I\notin F\). On the other hand, a basic product
neighbourhood of \(\top_I\) restricts only finitely many
coordinates. Choose a nonempty finite set \(J\) containing all of
them. Then \(g_J\in F\) agrees with \(\top_I\) at every
restricted coordinate, and hence belongs to that neighbourhood.
Thus \(\top_I\) lies in the product closure of \(F\) but not in
\(F\), proving the strict inclusion.
\end{proof}

The preceding proposition gives a basic obstruction to Scott--Isbell
coincidence. The chain \(\mathbb N^{\mathrm{op}}\) is the simplest
instance of the downward well-ordered obstruction in
Theorem~\ref{thm:JJL-bicomplete}. Indeed, let \(I\) be an uncountable
set with the discrete topology. Then \(I\) is core compact and
\(C(I,\mathbb N^{\mathrm{op}})=(\mathbb N^{\mathrm{op}})^I\).
By Lemma~\ref{lem:alexandrov}, the Isbell topology on this function
space is the product topology. Proposition~\ref{prop:countable-prototype}
therefore gives
\(\Is(C(I,\mathbb N^{\mathrm{op}}))
\subsetneq\sigma(C(I,\mathbb N^{\mathrm{op}}))\).
Thus \(\SI(\mathbb N^{\mathrm{op}})\) fails.

\subsection{General obstructions}

The preceding subsection treats the basic countable obstructions.
For a general downward well-ordered chain \(W\), the same idea applies,
but the number of coordinates must be chosen according to the size of
a coinitial family in \(W\). A subset \(B\) of a chain \(W\) is \emph{coinitial} if, for every
\(w\in W\), some \(b\in B\) satisfies \(b\leq w\).

\begin{theorem}\label{thm:power-separation}
Let \(W\) be a downward well-ordered chain without a least element.
There is a set \(I\) such that
\[
  \tprod(W^I)\subsetneq\sigma(W^I).
\]
More precisely, there is a Scott-closed set \(F\subseteq W^I\) whose
product closure contains \(\top_I\), although \(\top_I\notin F\).
\end{theorem}

\begin{proof}
Let \(\kappa\) be the least cardinality of a coinitial subset of \(W\).
Since \(W^{\mathrm{op}}\) is order-isomorphic to a limit ordinal of
cofinality \(\kappa\), the standard characterization of ordinal
cofinality
yields a strictly decreasing coinitial chain
\(
  (w_\xi)_{\xi<\kappa}
\)
in \(W\), indexed by \(\kappa\).

Let
\(
  I=\kappa^+
\) be the least cardinal strictly larger than \(\kappa\),
regarded as its initial ordinal. 
For every \(j<I\), we have
\(|j|\leq\kappa\), so choose an injection
\(
  r_j:j\longrightarrow\kappa.
\)

We first record the property of these maps that will be needed.
Let \(E\subseteq I\) have cardinality \(\kappa^+\). Choose \(j\in E\)
such that
\(
  |E\cap j|=\kappa.
\)
Since \(r_j\) is injective,
\(
  |r_j[E\cap j]|=\kappa.
\)
Hence \(r_j[E\cap j]\) is unbounded in \(\kappa\).
Since \((w_\xi)_{\xi<\kappa}\) is decreasing and coinitial, it follows
that
\begin{equation}\label{eq:large-subset-coinitial}
  \{w_{r_j(i)}:i\in E\cap j\}
  \quad\text{is coinitial in }W.
\end{equation}
In particular, for every \(u\in W\), there exist \(i<j\) in \(E\)
such that
\(
  w_{r_j(i)}<u.
\)

Let \(\top\) denote the greatest element of \(W\). For \(w<\top\), define
\[
  F_w=
  \left\{
    f\in W^I:
    f(i)>w,\ f(j)>w
    \Longrightarrow
    w_{r_j(i)}\geq w
    \text{ for all } i<j
  \right\},
\]
and put
\(
  F=\bigcup_{w<\top}F_w.
\)

We first show that each \(F_w\) is Scott closed. Indeed,
\[
  F_w=
  \bigcap_{\substack{i<j\\ w_{r_j(i)}<w}}
  \left(
    \{f\in W^I:f(i)\leq w\}
    \cup
    \{f\in W^I:f(j)\leq w\}
  \right).
\]
For every coordinate \(i\), the set
\(
  \{f\in W^I:f(i)\leq w\}
\)
is Scott closed. Hence \(F_w\) is Scott closed. In particular,
each \(F_w\), and therefore \(F\), is a lower set.

We next show that every \(f\in F\) belongs to \(F_w\) for a least
possible value of \(w\). Put
\[
  A(f)=\{w<\top:f\in F_w\}.
\]
By the definition of \(F\), the set \(A(f)\) is nonempty.

We claim that \(A(f)\) is bounded below. Since
\((w_\xi)_{\xi<\kappa}\) is coinitial in \(W\),
\[
  I
  =
  \bigcup_{\xi<\kappa}
  \{i\in I:f(i)>w_\xi\}.
\]
As \(|I|=\kappa^+\), there exists \(\xi_0<\kappa\) such that
\(
  E=\{i\in I:f(i)>w_{\xi_0}\}
\)
has cardinality \(\kappa^+\). 

We show that
\[
  A(f)\cap\da w_{\xi_0}=\varnothing.
\]
Indeed, let \(w\leq w_{\xi_0}\). By
\eqref{eq:large-subset-coinitial}, applied to \(E\), there exist
\(i<j\) in \(E\) such that
\(
  w_{r_j(i)}<w.
\)
On the other hand,
\[
  f(i)>w_{\xi_0}\geq w,
  \qquad
  f(j)>w_{\xi_0}\geq w.
\]
Thus \(f\notin F_w\). Hence \(A(f)\) is bounded below.

Let
\(
  a=\inf A(f).
\)
This infimum exists because the set of lower bounds of \(A(f)\) is
nonempty and \(W\) is downward well ordered. 
We now show that
\(f\in F_a\).

Suppose \(i<j\) and
\(
  f(i)>a,
  \
  f(j)>a.
\)
Set
\(
  p=\min\{f(i),f(j)\}.
\)
Then \(p>a\). 
Since \(a=\inf A(f)\), there exists \(w\in A(f)\) with
\(
  a\leq w<p;
\)
otherwise \(p\) would be a lower bound of \(A(f)\) strictly larger than \(a\). Since \(f\in F_w\), we obtain
\(
  w_{r_j(i)}\geq w\geq a.
\)
Therefore \(f\in F_a\). Thus
\(
  a=\min A(f).
\)
For later use, write
\(
  \mu(f)=\min A(f).
\)

We now prove that \(F\) is closed under directed suprema.
Let \(D\subseteq F\) be nonempty and directed. Since \(W\) is
downward well ordered, the nonempty set
\(
  \{\mu(d):d\in D\}
\)
has a greatest element. Choose \(d_0\in D\) such that
\[
  a:=\mu(d_0)
  =
  \max\{\mu(d):d\in D\}.
\]
We claim that
\(
  D\subseteq F_a.
\)
Let \(d\in D\). Choose \(e\in D\) with
\(
  d\leq e
  \)  and \(
  d_0\leq e.
\)
Since every \(F_w\) is a lower set,
\(
  A(e)\subseteq A(d)\cap A(d_0).
\)
Consequently,
\(
  \mu(e)\geq\mu(d_0)=a.
\)
By the choice of \(a\),
\(
  \mu(e)\leq a.
\)
Hence
\(
  \mu(e)=a.
\)
Thus
\(
  a\in A(e)\subseteq A(d),
\)
and therefore \(d\in F_a\). This proves \(D\subseteq F_a\).

Since \(F_a\) is Scott closed,
\(
  \sup D\in F_a\subseteq F.
\)
Thus \(F\) is Scott closed.

It remains to show that \(F\) is not closed in the product topology.
Let
\(
  \top_I\in W^I
\)
be the constant function with value \(\top\).

First,
\(
  \top_I\notin F.
\)
Indeed, fix \(w<\top\). Applying
\eqref{eq:large-subset-coinitial} to \(E=I\), we can choose \(i<j\)
such that
\(
  w_{r_j(i)}<w.
\)
Since
\(
  \top_I(i)=\top_I(j)=\top>w,
\)
we have
\(
  \top_I\notin F_w
\)
for every \(w<\top\), and therefore
\(
  \top_I\notin F.
\)

On the other hand, \(\top_I\) belongs to the product closure of \(F\).
Let \(J\subseteq I\) be finite. If \(J\) contains at least two elements, choose \(v_J<\top\) below all the finitely many elements
\(
  w_{r_j(i)} \) for \(
   i<j \in J.
\)
If \(|J|<2\), choose any \(v_J<\top\). Define \(g_J\in W^I\) by
\[
  g_J(i)=
  \begin{cases}
    \top, & i\in J,\\
    v_J, & i\notin J.
  \end{cases}
\]
If
\(
  g_J(i)>v_J
  \ \text{and} \
  g_J(j)>v_J,
\)
then \(i,j\in J\), and hence
\(
  w_{r_j(i)}>v_J.
\)
Therefore
\(
  g_J\in F_{v_J}\subseteq F.
\)
Moreover,
\(
  g_J(i)=\top_I(i) \) for all \(
  i\in J.
\)

Every basic product neighbourhood of \(\top_I\) restricts only
finitely many coordinates. Choosing \(J\) to contain those
coordinates shows that every such neighbourhood meets \(F\).
Consequently,
\(
  \top_I\in\overline{F}^{\,\tprod},
\)
while \(\top_I\notin F\). Thus \(F\) is not product closed.

Since \(F\) is Scott closed but not product closed, 
\(
  \tprod(W^I)\subsetneq\sigma(W^I).
\)
\end{proof}

For a nonempty set \(I\), let \(X_I=\{p,q\}\cup\{d_i:i\in I\}\)
carry the Alexandrov topology of the order whose only strict comparisons
are \(p<d_i\) and \(q<d_i\), for \(i\in I\).
The space \(X_I\) is core compact by Lemma~\ref{lem:alexandrov}.
It is also compact: any open set containing \(p\), together with any
open set containing \(q\), covers \(X_I\).
The following proposition transfers the separation on a power to a
space of continuous functions.

\begin{proposition}\label{prop:fan-transfer}
Let \(R\) be a dcpo, let \(u,v\in R\) be compact, and put
\(M=\up u\cap\up v\). If
\(\tprod(M^I)\subsetneq\sigma(M^I)\) for a nonempty set \(I\),
then \(\Is(C(X_I,R))\subsetneq\sigma(C(X_I,R))\).
\end{proposition}

\begin{proof}
Since \(u,v\) are compact, \(M\) is Scott open in \(R\), and
\(H=\{f\in C(X_I,R):f(p)\geq u,\ f(q)\geq v\}\) is Scott open
in \(C(X_I,R)\). Their intrinsic Scott topologies are therefore the
corresponding subspace topologies. For \(f\in H\), monotonicity gives
\(f(d_i)\in M\). Define \(r:H\to M^I\) by
\(r(f)(i)=f(d_i)\), and define \(s:M^I\to H\) by
\(s(g)(p)=u\), \(s(g)(q)=v\), and \(s(g)(d_i)=g(i)\), for
\(g\in M^I\). Each \(s(g)\) is monotone and hence continuous.
Thus the maps are well defined and \(r\circ s=\id_{M^I}\). The map \(r\) is monotone and preserves
pointwise directed suprema, so it is Scott continuous.

To check continuity of \(s\) from the product topology to the relative
pointwise topology, let
\(B=H\cap\bigcap_{k=1}^n\{f\in C(X_I,R):f(y_k)\in V_k\}\)
be a basic open set of \(H\), with \(y_k\in X_I\) and
\(V_k\in\sigma(R)\). Since \(s(M^I)\subseteq H\),
\(s^{-1}(B)=\bigcap_{k=1}^n\{g\in M^I:s(g)(y_k)\in V_k\}\).
For \(y_k=p\) or \(q\), the corresponding factor is either \(M^I\)
or empty. For \(y_k=d_i\), it is \(\pi_i^{-1}(M\cap V_k)\),
where \(\pi_i:M^I\to M\) is the coordinate projection. This is
product open because \(M\cap V_k\in\sigma(M)\). Thus
\(s^{-1}(B)\) is product open, proving continuity of \(s\).

Choose \(O\in\sigma(M^I)\setminus\tprod(M^I)\). Since \(r\) is
Scott continuous, \(r^{-1}(O)\) is Scott open in \(H\), and hence
in \(C(X_I,R)\) because \(H\) itself is Scott open. If it were
Isbell open, Lemma~\ref{lem:alexandrov} would make it pointwise open.
Continuity of \(s\) would then make
\(O=s^{-1}(r^{-1}(O))\) product open, a contradiction.
Together with the inclusion
\(\Is(C(X_I,R))\subseteq\sigma(C(X_I,R))\) from
Proposition~\ref{prop:isbell-below-scott}, this proves the claim.
\end{proof}

Then we obtain the following three required obstructions.

\begin{proposition}\label{prop:standard-detectors}
Let \(W\) be a downward well-ordered chain without a least element.
Then \(\SI(W)\) fails. Moreover, both
\(\SIcpt(K(W))\) and \(\SIcpt(K(W)_\bot)\) fail, with compact
core-compact Alexandrov spaces witnessing the latter two failures.
\end{proposition}

\begin{proof}
Choose \(I\) as in Theorem~\ref{thm:power-separation}. Give \(I\) the
discrete topology. Then \(C(I,W)=W^I\), and
Lemma~\ref{lem:alexandrov} identifies the Isbell topology with the
product topology. Hence \(\SI(W)\) fails.

Now let \(R=K(W)\) or \(K(W)_\bot\). Every directed subset of \(R\)
has a greatest element: if it meets \(W\), the greatest element of its
intersection with \(W\) is greatest in the whole set; otherwise it is
a finite directed set. Hence every element of \(R\) is compact.
Moreover, \(\up a\cap\up b=W\). Therefore
Proposition~\ref{prop:fan-transfer}, together with
Theorem~\ref{thm:power-separation}, gives the required witness \(X_I\).
\end{proof}

Finally, we record the finite obstructions needed for the compact-source
classification. In this case countably many coordinates already
suffice.

\begin{corollary}\label{cor:finite-detectors}
The properties \(\SIcpt(\mathsf B_4^\top)\) and
\(\SIcpt(\mathsf B_4)\) both fail.
\end{corollary}

\begin{proof}
Let \(R\) be either \(\mathsf B_4^\top\) or \(\mathsf B_4\), and put
\(M=\up u\cap\up v\). Thus \(M=\{a,b,\top\}\) in the first
case and \(M=\{a,b\}\) in the second.

Let \(z\in M^{\mathbb N}\) be constantly \(a\), and let
\(z_n\) agree with \(z\) except at coordinate \(n\),
where it has value \(b\). The set
\(\{z_n:n\in\mathbb N\}\) is Scott closed: its elements are
minimal in \(M^{\mathbb N}\), and every directed subset of it is a
singleton. It omits \(z\), but every basic product
neighbourhood of \(z\) contains some \(z_n\). Hence
\[
  \tprod(M^{\mathbb N})\subsetneq\sigma(M^{\mathbb N}).
\]
Since \(u\) and \(v\) are compact,
Proposition~\ref{prop:fan-transfer} gives the conclusion.
\end{proof}

\section{The three classifications}
\label{sec:classifications}

\subsection{All core-compact sources}

The first classification follows once the three non-bicomplete retracts have
been excluded.

\begin{proof}[Proof of Theorem~\ref{thm:main-corecompact}]
Suppose that $L$ is bounded complete. Since $L$ is continuous, it is
bicomplete, and Theorem~\ref{thm:xi-yang}(i) gives $\SI(L)$.

Conversely, assume $\SI(L)$. The continuous dcpo $L$ is sober and meet
continuous. If it were not bicomplete,
Theorem~\ref{thm:JJL-bicomplete}(i) would give a Scott-continuous retract
\[
  R\in\{W,K(W),K(W)_\bot\}
\]
for some downward well-ordered chain $W$ without a bottom element.
Proposition~\ref{prop:target-inheritance}(i) would imply $\SI(R)$, contrary
to Proposition~\ref{prop:standard-detectors}. Hence
$L$ is bicomplete, and Theorem~\ref{thm:xi-yang}(i) now shows that $L$ is
bounded complete.
\end{proof}

\subsection{All compact core-compact sources}

For compact sources, the absence of a least element is separated from the
remaining failures of conditional bounded completeness by passing to a
lifting.

\begin{lemma}\label{lem:lifting-cbc}
Let $L_\bot$ be obtained from a dcpo $L$ by adjoining a new least element.
Then $L$ is conditionally bounded complete if and only if $L_\bot$ is bounded
complete. If $L$ is continuous, then $L_\bot$ is continuous.
\end{lemma}

\begin{proof}
If $L$ is conditionally bounded complete and $A\subseteq L_\bot$ has an
upper bound, then $\sup A=\bot$ when $A$ is empty or contained in
$\{\bot\}$; otherwise the supremum of $A\cap L$ in $L$ is also the supremum
of $A$ in $L_\bot$. Thus $L_\bot$ is bounded complete.

Conversely, if $L_\bot$ is bounded complete and $A\subseteq L$ is nonempty
and upper bounded in $L$, its supremum in $L_\bot$ cannot be the newly
adjoined bottom and therefore lies in $L$. Preservation of continuity by
lifting is standard \cite[Section~3.2.5]{AbramskyJung1994}.
\end{proof}

\begin{proof}[Proof of Theorem~\ref{thm:main-compact}]
Suppose first that $L$ is conditionally bounded complete. By
Lemma~\ref{lem:lifting-cbc}, $P=L_\bot$ is continuous and bounded complete.
Theorem~\ref{thm:main-corecompact} gives $\SI(P)$. The sub-dcpo
$L=P\setminus\{\bot\}$ is Scott open, so part~(ii) of
Proposition~\ref{prop:target-inheritance} gives $\SIcpt(L)$.

Conversely, assume $\SIcpt(L)$ and put $P=L_\bot$. Suppose, towards a
contradiction, that $L$ is not conditionally bounded complete. Then $P$ is
pointed, continuous, sober, and not bounded complete.

We shall use the following construction for each pointed retract of $P$.
Suppose that $R$ is pointed and that
\[
  R\mathrel{\mathop{\rightleftarrows}^{\iota}_{\rho}}P,
  \qquad \rho\circ\iota=\id_R,
\]
is a Scott-continuous retraction. Put
\[
  S=R\setminus\{\bot_R\},
  \qquad
  U=\rho^{-1}(S).
\]
Since $\bot_P\leq\iota(\bot_R)$, one has
$\rho(\bot_P)=\bot_R$. Hence $U$ is a Scott-open sub-dcpo of $L$, and
$\iota(S)\subseteq U$. The restrictions of $\rho$ and $\iota$ exhibit $S$
as a Scott-continuous retract of $U$. Proposition~\ref{prop:target-inheritance}
therefore gives
\begin{equation}\label{eq:punctured-transfer}
  \SIcpt(L)\Longrightarrow\SIcpt(S).
\end{equation}

If $P$ were not bicomplete, Theorem~\ref{thm:JJL-bicomplete}(ii) would give a
retract $K(W)_\bot$ of $P$. Applying \eqref{eq:punctured-transfer} with
$S=K(W)$ contradicts Proposition~\ref{prop:standard-detectors}. Thus $P$ is
bicomplete.

If $P$ is not an $L$-dcpo, Theorem~\ref{thm:JJL-nonL} and
\eqref{eq:punctured-transfer} imply $\SIcpt(\mathsf B_4^\top)$, contrary to
Corollary~\ref{cor:finite-detectors}. If $P$ is an $L$-dcpo, then
Theorem~\ref{thm:AJ-L-not-BC} and the same transfer imply
$\SIcpt(\mathsf B_4)$, again contradicting
Corollary~\ref{cor:finite-detectors}. Hence $P$ is bounded complete, and
Lemma~\ref{lem:lifting-cbc} shows that $L$ is conditionally bounded complete.
\end{proof}

\subsection{All RW-spaces}

For RW-spaces we use Alexandrov detectors with a greatest point.
For a nonempty set $I$, define the capped star and capped two-root fan by
\[
  S_I=\{d_i:i\in I\}\cup\{t\},\qquad d_i<t,
\]
and
\[
  F_I=\{p,q,t\}\cup\{d_i:i\in I\},\qquad p,q<d_i<t.
\]
Both carry their Alexandrov topologies.

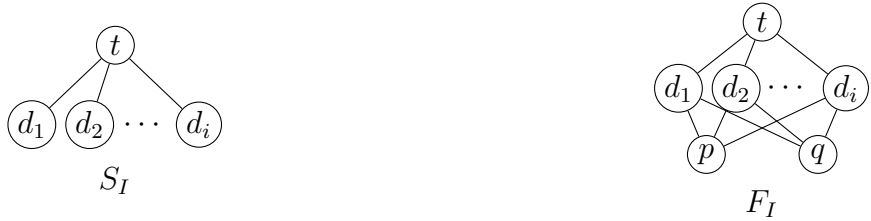
\begin{figure}[!htbp]
\centering
\begin{minipage}{0.44\linewidth}
\centering
\begin{tikzpicture}[x=0.85cm,y=0.75cm,
 every node/.style={circle,draw,inner sep=1.1pt,minimum size=5mm}]
\node (d1) at (-1.3,0) {$d_1$};
\node (d2) at (-0.4,0) {$d_2$};
\node[draw=none] at (0.45,0) {$\cdots$};
\node (di) at (1.3,0) {$d_i$};
\node (t) at (0,1.4) {$t$};
\draw (d1)--(t) (d2)--(t) (di)--(t);
\end{tikzpicture}

$S_I$
\end{minipage}
\hfill
\begin{minipage}{0.52\linewidth}
\centering
\begin{tikzpicture}[x=0.82cm,y=0.70cm,
 every node/.style={circle,draw,inner sep=1.1pt,minimum size=5mm}]
\node (p) at (-0.9,0) {$p$};
\node (q) at (0.9,0) {$q$};
\node (d1) at (-1.35,1.25) {$d_1$};
\node (d2) at (-0.42,1.25) {$d_2$};
\node[draw=none] at (0.42,1.25) {$\cdots$};
\node (di) at (1.35,1.25) {$d_i$};
\node (t) at (0,2.5) {$t$};
\foreach \x in {d1,d2,di}{\draw (p)--(\x); \draw (q)--(\x); \draw (\x)--(t);}
\end{tikzpicture}

$F_I$
\end{minipage}
\caption{The capped Alexandrov detectors for the RW-space classification.}
\label{fig:rw-detectors}
\end{figure}

The capped spaces have the RW property and retain the power obstruction.

\begin{proposition}\label{prop:rw-detectors}
Let $W$ be a downward well-ordered chain without a bottom element. There are
connected RW-spaces witnessing the failure of $\SIrw(W)$,
$\SIrw(K(W))$, and $\SIrw(K(W)_\bot)$.
\end{proposition}

\begin{proof}
The Alexandrov spaces $S_I$ and $F_I$ are core compact by
Lemma~\ref{lem:alexandrov}. Every nonempty open set in either space
contains $t$. Hence both spaces are connected, and
\cite[Proposition~2.10]{KouLuo2003} shows that they are RW-spaces.

Choose $I$ as in Theorem~\ref{thm:power-separation}. For $R=W$, let
$Z=\{d_i:i\in I\}$ be discrete and put $Y=S_I$. For $R=K(W)$ or
$K(W)_\bot$, put $Z=X_I$ and $Y=F_I$. In each case, $Y$ is obtained
from $Z$ by adjoining a greatest point $t$, and $R$ has a greatest element
$\top_R$.

Restriction gives a map
\[
  r:C(Y,R)\to C(Z,R),\qquad r(g)=g|_Z,
\]
with section $s$ defined by
\[
  s(f)|_Z=f,\qquad s(f)(t)=\top_R.
\]
Both maps are Scott continuous and pointwise continuous, and
$r\circ s=\id_{C(Z,R)}$. If the Scott and pointwise topologies agreed on
$C(Y,R)$, then, for every Scott-open $O\subseteq C(Z,R)$, the identity
$O=s^{-1}(r^{-1}(O))$ would make $O$ pointwise open. By
Lemma~\ref{lem:alexandrov}, this would imply coincidence on $C(Z,R)$,
contrary to the witnesses in Proposition~\ref{prop:standard-detectors}.
Since $Y$ is a connected RW-space, the result follows.
\end{proof}

\begin{proof}[Proof of Theorem~\ref{thm:main-rw}]
Assume first that $\SIrw(L)$ holds. If $L$ were not bicomplete,
Theorem~\ref{thm:JJL-bicomplete}(i) would yield a Scott-continuous retract
\[
  R\in\{W,K(W),K(W)_\bot\}.
\]
Proposition~\ref{prop:target-inheritance}(i) would imply $\SIrw(R)$, contrary
to Proposition~\ref{prop:rw-detectors}. Thus $L$ is bicomplete, and
Theorem~\ref{thm:xi-rw} shows that $L$ is a pointed continuous
$L$-domain.

Conversely, let $L$ be a pointed continuous $L$-domain. Every filtered subset
has an infimum: after choosing one member $d_0$, its part below $d_0$ is
downward cofinal and has an infimum in the complete lattice $\da d_0$.
Thus $L$ is bicomplete, and Theorem~\ref{thm:xi-rw} gives $\SIrw(L)$.
\end{proof}

\subsection{A further question}

For a continuous dcpo $L$, let $\SIcoh(L)$ mean coincidence for every
core-compact coherent source. A \emph{bifinite domain} is a pointed dcpo
admitting a directed family of finite-image idempotent Scott-continuous maps
$d\leq\id$ whose pointwise supremum is the identity
\cite[Theorem~4.2.6]{AbramskyJung1994}. A \emph{continuous B-domain} is a
Scott-continuous retract of a bifinite domain. Xi and Liang proved that every
continuous B-domain satisfies $\SIcoh(L)$ \cite{XiLiang2009}.

The Alexandrov detector used in Proposition~\ref{prop:fan-transfer} is not
coherent when $I$ is infinite: the compact saturated sets $\up p$ and
$\up q$ have intersection $\{d_i:i\in I\}$, which is covered by the
open singletons $\{d_i\}$ and has no finite subcover. Thus the detector
method above does not determine the maximal target class for coherent
sources.

\begin{question}\label{q:coherent}
Characterize the continuous dcpos $L$ satisfying $\SIcoh(L)$. 
\end{question}

\end{document}